\documentclass[11pt,twoside]{amsart}
\usepackage{pgfplots}
\pgfplotsset{compat=1.15}
\usepackage{mathrsfs}
\usetikzlibrary{arrows}
\usepackage{amsmath, amssymb, amsfonts, amsthm,latexsym,,mathtools,,hyperref,fancyhdr,comment, tabularht, tabularx, longtable, mathrsfs, cleveref,booktabs, color, geometry,bm}
\usepackage{enumitem}
\usepackage{multicol}
\usepackage{hyperref,url}
\usepackage{tikz}
\tikzstyle{vertex}=[circle, draw, inner sep=0pt, minimum size=6pt,fill=black]

\long\def\symbolfootnote[#1]#2{\begingroup%
	\def\thefootnote{\fnsymbol{footnote}}\footnote[#1]{#2}\endgroup}

\allowdisplaybreaks
\def \Z {{\mathbb{Z}}}

\def \F {{\mathbb{F}}}

\def \L {{\mathbb{L}}}

\def \U {{\mathcal{U}}}

\newtheorem*{theorem*}{Theorem}
\newtheorem{theorem}{Theorem}[section]

\newtheorem{lemma}[theorem]{Lemma}

\newtheorem*{ex*}{Example}

\date{}

\begin{document}
\title{Solvability of groups via cyclic subgroup count}
	\author{Angsuman Das}
	\address{ Presidency University, Kolkata, India,}
	\email{angsuman.maths@presiuniv.ac.in}
	\author{Khyati Sharma}
	\address{Indian Institute of Science Education and Research, Berhampur, Odisha, India.}
	\email{khyatisharma0907@gmail.com}

	\today 

\begin{abstract}
   In this paper, we provide new criteria for the solvability and supersolvability of a finite group based on its number of cyclic subgroups. A finite group $G$ is called $n$-cyclic if it contains $n$ cyclic subgroups. This paper also partially extends the classification of $n$-cyclic groups for $n\geq 13$.
\end{abstract}

\maketitle
\section{Introduction}
A major goal in the study of finite groups is to understand how far the structure of a group can be recovered from the partial information attached to it. In many cases, quantities obtained by counting certain algebraic objects or by studying probabilistic features of the group can reveal important structural properties. This approach is very useful in finite group theory, mainly because such invariants are often much easier to handle than more complex information, such as the entire subgroup lattice or the character table.

Several important examples show the strength of this approach. For instance, the sum of element orders and the average element order were first introduced by Amiri, Amiri, and Isaacs~\cite{amiri2009sums}. Later, these quantities were further studied deeply in connection with detecting structural properties of finite groups~\cite{herzog2018two, azad2018criterion, tuarnuauceanu2020detecting, herzog2023new}. Another quantity considered in this context is the order sequence~\cite{cameron2023order}. Most recently, Lucchini~\cite{lucchini2026characterizing} characterized solvability using the nilpotent probability, that is, the probability that two randomly chosen elements of a group 
$G$ generate a nilpotent subgroup. In fact, one of the main objectives behind this study is to give new criteria for the solvability of finite groups. 

From this perspective, subgroup-based invariants arise quite naturally. A notable contribution in this direction is due to Das and Mandal~\cite{das2024solvability}, who obtained several criteria for solvability, supersolvability, and nilpotency of a finite group by using subgroup counting. Their work shows that subgroup counting, despite its elementary nature, can yield important structural conclusions. However, counting all subgroups is often difficult in practice, which makes it desirable to look for more accessible alternatives.

 A natural choice is to focus on cyclic subgroups. They are among the simplest subgroups of a finite group, yet their number often reflects meaningful information about the whole group~\cite{garonzi2018number}. Moreover, the number of cyclic subgroups satisfies useful extremal bounds. A result of Richards~\cite{richards1984remark} states that if $G$ is a finite group of order $n$, then $c(G)\geq d(n)$, where $c(G)$ denotes the number of cyclic subgroups of $G$ and $d(n)$ is the number of positive divisors of $n$. Moreover, equality holds if and only if $G$ is cyclic. On the other hand, it is easy to see that $c(G)\leq |G|$, and equality holds if and only if $G$ is an elementary abelian $2$-group. Thus, unlike the counting of all subgroups, the number of cyclic subgroups comes with explicit lower and upper bounds, making it more manageable.
 
This viewpoint has also led to several classification results. Recall that a finite group $G$ is called $n$-cyclic if it has exactly $n$ cyclic subgroups. Zhou~\cite{zhou2016finite}, Kalra~\cite{kalra2019finite}, and Ashrafi and Haghi~\cite{ashrafi2019n} classified all $n$-cyclic groups for $3\leq n\leq 10$. Reddy and Sharma~\cite{11cyclic,sharma2025groups} classified all $n$-cyclic groups for $n=11$ and $12$. Now, the natural question is: how is this counting helpful for understanding the structure of the underlying group? Motivated by the work of Das and Mandal~\cite{das2024solvability}, and by the broader theme of characterizing group-theoretic properties through numerical invariants, in this paper, we give new characterizations of solvability and supersolvability in terms of cyclic subgroup counting. We prove the following results in this article.

\vspace{0.2cm}
\noindent \textbf{Theorem~A.} Let $G$ be a finite group with $c(G)< 50$. Then either $G$ is solvable or $G\cong A_5$ or $SL(2,5)$. (Theorem~\ref{SG1}~-~\ref{SG4})

\vspace{0.2cm}
\noindent \textbf{Theorem~B.} Let $G$ be a finite group with $ c(G)\leq 17$. Then either $G$ is supersolvable or $G\cong A_4, {\Z_2}^2\rtimes \Z_9, SL(2,3), S_4, \Z_q\times A_4$, ${\Z_2}^3\rtimes \Z_7$ or $(\mathbb{Z}_2\times \mathbb{Z}_2)\rtimes \mathbb{Z}_{27}$, where $q$ is a prime number. (Theorem~\ref{thm:SSG} and Lemma~\ref{lem:SSG})

\vspace{0.2cm}
The organization of the article is as follows: In Section~\ref{sec:notations}, we set up notations, recall some results, and formulate lemmas on which our proofs are based. In Section~\ref{sec:Solvability} and Section~\ref{sec:supersolvability}, we give criteria for solvability and supersolvability of a finite group based on its number of cyclic subgroups.

\section{Notations and Preliminaries}\label{sec:notations}
Throughout this paper, all groups are finite. The notations $\Z_n$, $D_{2n}$, $Q_{2^n}$, $SL(n,q)$ and $A_{n}$ denote the cyclic group of order $n$, dihedral group of order $2n$, generalized quaternion group of order $2^n$, special linear group of $n\times n$ matrices over the field $\F_q$ and alternating group of degree $n$ respectively. The projective special linear group of $n\times n$ matrices over the field $\F_{q}$ is denoted by $PSL(n,q)=L_n(q)$. The number of Sylow $p$-subgroups of $G$ is denoted by $n_p(G)$. If the group $G$ is fixed, then we simply use $n_p$ to denote the number of Sylow $p$-subgroups of $G$. Let $d(n), \pi(n)$ denote the number of positive divisors and the number of distinct prime divisors of $n$, respectively. For a group $G$, $C(G)$ is the collection of all cyclic subgroups of $G$, and $c(G)$ denotes the number of cyclic subgroups of $G$. A group $G$ is called {\em minimal simple group} if $G$ is simple and all its proper subgroups are solvable. In this paper, all the calculations for small groups are done by using \textsf{GAP}~\cite{GAP4} and the {\em SmallGroup$(n, i)$} denotes the $i^{th}$ group of order $n$ in the {\em Small Group Library} of \textsf{GAP}.
%\begin{lemma}[\cite{rose2009course}]\label{lem:2.1}
%\textcolor{red}{DELETE} Let $G$ be a finite group. Then
%\begin{enumerate}
    %\item $G$ is solvable if and only if $G$ has a solvable normal subgroup $H$ such that the quotient $G/H$ is also solvable.
%\item $G$ is supersolvable if and only if $G$ has a cyclic normal subgroup $H$ such that the quotient $G/H$ is supersolvable. 
%\end{enumerate}
%\end{lemma}
\begin{lemma}[\cite{miller1929number}]\label{lem:2.2}
 Let $G$ be a finite group. Then

$$|G|=\sum\limits_{m||G|}c(m)\varphi (m),$$
where $c(m)$ denotes the number of cyclic subgroups of order $m$ in $G$.   
\end{lemma}
\begin{lemma}[\cite{MR2963406}]\label{lem:2.3} If $G_1$ and $G_2$ are two groups with $gcd(|G_1|,|G_2|)=1$, then $$c(G_1\times G_2)=c(G_1)\times c(G_2).$$
    \end{lemma}
    %\begin{lemma}[\cite{dummit1991abstract}]\label{lem:2.4}
    %\textcolor{red}{DELETE} If $G$ is a finite group and $H$ is a subgroup of $G$, then the number of distinct conjugates of $H$ is $[G:N_G(H)]$.
%\end{lemma}
\begin{lemma}\label{lem:2.5}
    If $G$ is a finite group and $H$ is a normal subgroup of $G$, then
    \begin{enumerate}
        \item $c(G/H)\leq c(G)$
        \item $c(G)\geq c(G/H)+c(H)-1$
    \end{enumerate}
\end{lemma}
\begin{proof}
    Consider the map
\[
\phi : C(G) \longrightarrow C(G/H)
\]
defined by
\[
\phi(\langle g \rangle) = \langle gH \rangle
\quad \text{for each } g \in G,
\]
where \(C(G)\) denotes the collection of all cyclic subgroups of \(G\). Now it is easy to see that the map $\phi$ is surjective. Therefore $c(G/H)\leq c(G)$.\\
For every element $g\in G$, there exists a cyclic subgroup in $G/H$ generated by $gH$. Also, if $g\in H$, then the subgroup of $G/H$ generated by $gH$ is trivial. This gives $c(G)\geq c(G/H)+c(H)-1$.
\end{proof}
\begin{lemma}\label{lem:SSG}(\cite{ashrafi2019n,zhou2016finite,kalra2019finite,11cyclic,sharma2025groups})
    Let $G$ be a finite group with $c(G)\leq 12$. Then either $G$ is supersolvable or $G\cong A_4$ or ${\Z_2}^2\rtimes \Z_9=${\em SmallGroup$(36,3)$}.
\end{lemma}
\begin{lemma}\label{ino:even}
    The number of involutions in $PSL(2,q)$, where $q=2^m$ is $q^2-1$.
\end{lemma}
\begin{proof}
   Since $q=2^m$ is even, the field $\F_q$ has characteristic $2$. Also $-1=1$ in $\F_q$. Therefore, $Z(SL(2,q))=\{I\}$ and $PSL(2,q)\cong SL(2,q)$. Thus, to count involutions of $PSL(2,q)$, it is sufficient to count the number of involutions of $SL(2,q)$. Let $A\in SL(2,q)$ be an involution. Then $A^2=I$ and $A\not =I$. Since characteristic of $\F_q$ is $2$, then $x^2-1=(x-1)^2$ and the minimal polynomial of $A$ divides $(x-1)^2$. This implies that all the eigenvalues of $A$ are $1$ and $A$ is unipotent. Also, as $A\not =I$, the matrix $A-I$ is non-zero and nilpotent, so it is conjugate to the matrix $\begin{pmatrix}
0 & 1\\
0 & 0
\end{pmatrix}$ and $A$ is conjugate to $$I+\begin{pmatrix}
0 & 1\\
0 & 0
\end{pmatrix}=\begin{pmatrix}
1 & 1\\
0 & 1
\end{pmatrix}.$$ Therefore, every non-identity involution is conjugate to $B=\begin{pmatrix}
1 & 1\\
0 & 1
\end{pmatrix}$ and all the involutions form a single conjugacy class. We now compute the centralizer of $B$ in $SL(2,q)$. $$\text{Let}~ X=\begin{pmatrix}
a & b\\
c & d
\end{pmatrix} \in SL(2,q)$$ which satisfies $XB=BX$. Then $$XB=\begin{pmatrix}
a & a+b\\
c & c+d
\end{pmatrix} \text{and}~ BX=\begin{pmatrix}
a+c & b+d\\
c & d
\end{pmatrix}.$$ Comparing the terms both sides, we get $c=0$ and $a=d$ and $X=\begin{pmatrix}
a & b\\
0 & a
\end{pmatrix}.$ The determinant of $X$ is $1$, so $a^2=1$. This implies that $a=1$ as the $char(\F_q)=2$. Therefore $$C_{SL(2,q)}(B)=\bigg\{\begin{pmatrix}
1 & b\\
0 & 1
\end{pmatrix}, \text{where}~b\in \F_q\bigg\},$$ and $|C_{SL(2,q)}(B)|=q$. Moreover, $|SL(2,q)|=q(q^2-1)$, so the size of conjugacy class of $B$ is $$\frac{|SL(2,q)|}{|C_{SL(2,q)}(B)|}=\frac{q(q^2-1)}{q}.$$ Since all the non-identity involutions are conjugate to $B$. Hence $PSL(2,q)$, where $q=2^m$ contains $q^2-1$ involutions.
\end{proof}
\begin{lemma}(Lemma 2.9,~\cite{das2026group})\label{inv:odd}
The number of involutions in $PSL(2,q)$ for an odd prime power $q$ is 
$$\frac{q(q+1)}{2} \text{ if } q \equiv 1 \pmod{4}, \quad \frac{q(q-1)}{2} \text{ if } q \equiv 3 \pmod{4}.$$
    \end{lemma}
\begin{lemma}(Theorem 1 and 2,~\cite{herzog1968finite})\label{lem:K_3} 
If $G$ is a finite simple group whose order is divisible by exactly three primes, then $G$ is isomorphic to one of the following $A_5,A_6,L_2(7),L_2(8),\\L_2(17),L_3(3),U_3(3),U_4(2)$.
        \end{lemma}
        \begin{lemma}(Theorem 2,~\cite{shi1991simple})\label{lem:K_4}
        If $G$ is a finite simple group whose order is divisible by exactly four primes, then $G$ is isomorphic to  one of the following 
        \begin{enumerate}
            \item $A_7,A_8,A_9,A_{10},M_{11},M_{12},J_{2},L_2(16),L_2(25),L_2(49),L_2(81),L_3(4),L_3(5),L_3(7),L_3(8),\\L_3(17),L_4(3),S_4(4),S_4(5),S_4(7),S_4(9),S_6(2),O_8^{+}2,G_2(3),U_3(4),U_3(5),U_3(7),U_3(8),\\U_3(9),U_4(3),U_5(2),Sz(8),Sz(32),^{3}D_4(2),^{2}{F_4(2)'}.$
            \item $L_2(r)$, where $r$ is a prime and satisfies $r^2-1=2^a\cdot 3^b\cdot v^c$, where $a,b,c\geq 1, v>3$ and $v$ is a prime.
            \item $L_2(2^m)$, where $m$ satisfies $2^m-1=u$ and $2^m+1=3t^b$ with $m\geq 2$, $u,t$ are
primes, $t>3$ and $b\geq 1$.
\item $L_2(3^m)$, where $m$ satisfies $3^m+1=4t$ and $3^m-1=2u^c$ or $3^m+1=4t^b$ and $3^m-1=2u$, with $m\geq 2$, $u,t$ are odd primes, $b,c\geq 1$.
        \end{enumerate}
            
        \end{lemma}
\section{Results related to Solvability}\label{sec:Solvability}
In this section, we prove that if $1\leq c(G)\leq 49$ and $c(G)\not =32,49$, then $G$ is
solvable. We also classify the unique non-solvable groups with exactly $32$ or $49$ cyclic subgroups.
\begin{lemma}\label{lem:simple1}
    Let $G$ be a finite minimal simple group. Then $c(G)\geq 32$ and equality holds if and only if $G\cong PSL(2,4)$ or $A_5$.
\end{lemma}
\begin{proof}
    By using the classification of finite simple groups, $G$ is isomorphic to one of the following.
    \begin{enumerate}
        \item $PSL(2,2^p)$, where $p$ is a prime.
        \item $PSL(2,3^p)$, where $p$ is an odd prime.
    \item $PSL(2,p)$, where $p>3$ is a prime such that $5\mid p^2+1$.
    \item The Suzuki group $Sz(2^p)$, where $p$ is an odd prime.
    \item $PSL(3,3)$.
    \end{enumerate}
    If $G\cong PSL(2,2^p)$, where $p$ is a prime, then except for $p=2$ the number of involutions in $G\geq 63$ using Lemma~\ref{ino:even}. For $p=2$, $G=PSL(2,4)$ or $A_5$ and $c(G)=32$. Again if $G=PSL(2,3^p)$, where $p$ is an odd prime or $G=PSL(2,p)$, where $p>3$ is a prime such that $5\mid p^2+1$, then also either the number of involutions in $G$ are more than $32$ using Lemma~\ref{inv:odd} or $G\cong PSL(2,7)$. By simple calculation, we can check that $c(PSL(2,7))=79$. If $G=Sz(2^p)$, where $p$ is an odd prime, or $G=PSL(3,3)$ then the $|G|$ is $2^{2p}(2^{2p}+1)(2^p-1)$ or $2^4\cdot 3^3\cdot 5$ respectively. By Richard's theorem~\cite{richards1984remark} $c(G)>32$. Hence, the result holds.
\end{proof}
\begin{lemma}\label{lem:simple2}
The following statements hold for finite simple groups
\begin{enumerate}
    \item Among all finite non-abelian simple groups, $A_5$ has the minimum number of cyclic subgroups.
\item There does not exist a finite simple group $G$ with $c(G)=49$.
   \end{enumerate}
    \end{lemma}
    \begin{proof}
        Let $G$ be a finite simple group with the least number of cyclic subgroups. Then $c(G)\leq 32$ because $c(A_5)=32$. By Richard's theorem~\cite{richards1984remark} either $|G|$ has three prime divisors or $|G|=p^2qrs$, where $p,q,r$ and $s$ are distinct prime numbers. If $|G|$ has three prime divisors, then by Lemma~\ref{lem:K_3}, $G\in\{A_5,A_6,L_2(7),L_2(8),L_2(17),L_3(3),U_3(3),U_4(2)\}$. If $G\cong A_6$, then $c(G)=167$. If $G\cong L_2(7)$, then $|G|=168$ and $c(G)=79$. For $G\cong L_2(8)$ or $L_2(17)$, Lemma~\ref{ino:even} and~\ref{inv:odd} show that $G$ has more than $32$ involutions. If $G\cong L_3(3), U_3(3)$ or $U_4(2)$, then $|G|=2^4\cdot 3^3\cdot 13,~ 2^5\cdot 3^3\cdot 7$ and $2^6\cdot 3^4\cdot 5$ respectively. For these groups, using Richard's theorem~\cite{richards1984remark}, we can check that $c(G)>32$. If $G=p^2qrs$, then also by~\cite{walter1969characterization} and counting the involutions we get $c(G)>32$.

        Further assume that $G$ is a finite simple group, with $c(G)=49$. Then by Richard's theorem, either $|G|$ has three or four prime divisors or $|G|=p^2qrst$, where $p,q,r,s$ and $t$ are distinct prime numbers. If the order of $G$ has $3$ or $4$ prime divisors, then Lemma~\ref{lem:K_3} and~\ref{lem:K_4} give all such groups. For all such groups either by Richard's theorem~\cite{richards1984remark}, or by Lemma~\ref{ino:even}
        and~\ref{inv:odd}, we can see that $c(G)\not =49$. If $|G|=p^2qrst$, then also by using~\cite{walter1969characterization}, we can see that $c(G)\not =49$. Hence, the result follows.
    \end{proof}
\begin{theorem}\label{SG1}
    Let $G$ be a group such that $c(G)<32$, then $G$ is solvable.
\end{theorem}
\begin{proof}
    Let $G$ be a minimum ordered non-solvable group such that $c(G)<32$. We first show that $G$ must be simple.
    
    It is easy to see that all proper subgroups of $G$ are solvable, because $G$ is the minimum ordered group satisfying the above condition. If $H$ is a proper normal subgroup of $G$, then $G/H$ must be non-solvable. Now, by the minimality of $|G|$, we have $c(G/H)\geq 32$, which is a contradiction by Lemma~\ref{lem:2.5}. Thus, $G$ has no proper normal subgroup, and $G$ is simple.
    Among all finite simple groups, $A_5$ has the least number of cyclic subgroups by Lemma~\ref{lem:simple2}.  Therefore, $c(G)\geq 32$. Hence, the result holds. 
\end{proof}
\begin{theorem}\label{SG2}
    Let $G$ be a non-solvable group such that $c(G)=32$, then $G\cong A_5$.
\end{theorem}
\begin{proof}
    Since $G$ is non-solvable, then $G$ has subgroups $H$ and $N$ such that $N\trianglelefteq H$ and $H/N$ is a minimal simple group. If $c(H/N)\geq 33$, then $c(G)\geq 33$, which is a contradiction. Thus $c(H/N)=32$. If $H$ is a proper subgroup of $G$, then $c(G)\geq c(H)+1\geq c(H/N)+1=33$, which is a contradiction. Therefore $G=H$ and $c(G/N)=32$. If $N$ is a non-trivial subgroup of $G$, then $c(G)\geq c(G/N)+1=33$, which is also a contradiction. This implies that $N$ is trivial, and $G/N\cong G$ is a minimal simple group. Now, by Lemma~\ref{lem:simple1}, we get $G\cong A_5$.
\end{proof}
\begin{theorem}\label{SG3}
    Let $G$ be a group such that $33\leq c(G)\leq 48$, then $G$ is solvable.
\end{theorem}
\begin{proof}
Let $G$ be a non-solvable group of minimum order such that $33\leq c(G)\leq 48$. Then the claim is that $G$ is simple. Suppose $G$ is not simple, then $G$ has a proper non-trivial normal subgroup. Let $H$ be a proper non-trivial normal subgroup of $G$ of minimum order. Then either $H$ or $G/H$ is non-solvable. Now, we discuss these cases separately.

\textbf{Case~-~1:} If $H$ is non-solvable, then by the minimality of $|G|$, we have $c(H)<33$. By Theorems~\ref{SG1}, and~\ref{SG2}, $c(H)=32$ and $H\cong A_5$. Also, $G/H\cong G/A_5$ is a non-trivial group. Thus, $G/A_5$ has a subgroup $K/A_5$ of prime order, say $p$, such that $K$ is a non-solvable subgroup of $G$ and $|K|=60p$. If $p=2,3$ or $5$, then $|K|$ is $120, 180$ or $300$. Using \textsf{GAP}~\cite{GAP4}, we can verify that every non-solvable group of these orders has more than $48$ cyclic subgroups. Thus $p>5$, and the number of Sylow $p$-subgroups of $K$ is $n_p(K)=1+kp|60$. First assume that a Sylow $p$-subgroup of $K$ is not normal, then $p=7, 11, 19$ and $29$. If $p=19$ or $29$, then $c(G)>48$ as $c(H)=32$. Thus, $p$ is either $7$ or $11$ and $|K|=420$ or $660$. By \textsf{GAP}, we can check that these cases are not possible. Therefore, $K$ contains normal Sylow $p$-subgroup and $K= A_5\times \Z_p$. Since $(|A_5|,p)=1$, then by Lemma~\ref{lem:2.3} $c(K)=c(A_5)\times c(\Z_p)=64$, which is a contradiction.

\textbf{Case~-~2:} If $G/H$ is non-solvable, then by using similar arguments as the previous case, we get $c(G/H)=32$, $G/H\cong A_5$ and $|G|=60\times |H|$. If $|H|\in \{2,3,4,5\}$, then $|G|\in \{120,180,240,300\}$. It can be checked by using \textsf{GAP}~\cite{GAP4} that all non-solvable groups of these orders have at least $49$ cyclic subgroups. Therefore $|H|\not \in \{2,3,4,5\}$. Also, $H$ is the direct product of isomorphic simple groups by~\cite{hall2018theory}. If $H$ is the direct product of non-abelian simple groups, then $c(H)\geq 32$. This implies that $c(G)\geq c(G/H)+c(H)-1=63$, which is not possible. Therefore, $H$ is an elementary abelian $p$-group, that is $H=(\Z_p)^n$ and $c(H)=p^{n-1}+p^{n-2}+\dots+2$. Since $c(G)\geq c(G/H)+c(H)-1$ then $H\in \{(\Z_2)^4,(\Z_2)^3,(\Z_3)^3, (\Z_p)^2,\Z_q, \text{where}~ 3\leq p\leq 13,q\geq 7\}$. Now we discuss these cases separately.

\textbf{Subcase~-~1~}Let $H=(\Z_2)^4$, then $|G|=2^6\times 3\times 5$. Since $G/H=A_5$, the possible order of any element in $G/H$ is $1,2,3$ or $5$. Since $H$ is an elementary abelian $2$ group, so any non-trivial element of $H$ has order $2$. Let $xH\in G/H$ be any non-trivial element of $G/H$. Then $x^k\in H$, where $k=2,3$ or $5$. If $o(xH)=2$ in $G/H$, then $o(x)$ is either $2$ or $4$ in $G$. Thus if $xH\in G/H$ such that $x\not \in H$ then $o(x)\in \{2,3,4,5,6,10\}$. Also $c(H)=16$, then $c(G/H)+c(H)-1=47$. This implies that either $c(G)=47$ or $48$. If $c(G)=48$, then the maximum possible order for the $48$th cyclic subgroup of $G$, if it exists, is $30$. This is not possible by using Lemma~\ref{lem:2.2}.

    \textbf{Subcase~-~2~}If $H=(\Z_2)^3$, then $|G|=480$. By using \textsf{GAP}, we can check that no such non-solvable group of order $480$ has the number of cyclic subgroups lying in the range $33\leq c(G)\leq 48$. If $H=(\Z_3)^3$, then $|G|=2^2\times 3^4 \times 5$. Moreover, Sylow $3$-subgroup of $G$ is not normal, otherwise $G$ is solvable. By \textsf{GAP}, we can check that a Sylow $3$ subgroup, which contains $(\Z_3)^3$ as a subgroup, has at least $23$ cyclic subgroups. Also, the Sylow $3$-subgroup is not unique, so $G$ has at least $32$ cyclic subgroups of order $3^k$, where $0\leq k\leq 3$. Moreover, $G$ has at least $21$ cyclic subgroups corresponding to the subgroups of order $2$ and $5$ of $G/H$ whose orders are not equal to $3^k$, then $c(G)\geq 32+21>48$. Thus, this case is not possible.
    
    \textbf{Subcase~-~3~}If $H=\Z_p\times \Z_p$, where $3\leq p\leq 13$, then $|G|=60p^2$. First take $p=13$, in this case $|G|=2^2\times 3\times 5\times {13}^2$ and $G$ contains an abelian subgroup of order $5\times {13}^2$ say $M$ as $5$ does not divide ${13}^k-1$ for $k=1,2$. By \textsf{GAP}, we can check that $M$ has $13$ cyclic subgroups of order $65$. Thus $c(G)\geq 15+10+14+13$, where $15$ and $10$ denotes the cyclic subgroup of $G$ corresponding to cyclic subgroups of order $2$ and $3$ in $G/H$ respectively, $14$ denotes subgroups of order $13$ in $G$, and $13$ denotes subgroups of order $65$ in $G$. This implies that $c(G)>48$, which is a contradiction. Similarly, we can check that when $p=11$, then also $G$ contains a subgroup of order $3\times {11}^2$, which can not be unique otherwise $G$ will be solvable. This shows that $c(G)>48$. For $p=7$, also by following the similar steps, we can check that $c(G)>48$. If $p=5$, then $|G|=2^2\times 3\times 5^3$ and Sylow $5$-subgroup of $G$ can not be unique otherwise $G$, will be solvable. By Sylow's theorem, $G$ has at least $6$ subgroups of order $5^3$. Also, by using \textsf{GAP}, $G$ has at least $31$ cyclic subgroups of order $5^k$, where $k=1$ or $2$. Therefore, $c(G)\geq 15+10+31>48$, where $15$ and $10$ denotes the cyclic subgroups of $G$ corresponding to the subgroups of order $2$ and $3$ in $G/H$ respectively, and $31$ denotes the number of cyclic subgroups of order $5^k$ in $G$. For $p=3$, $|G|=540$, we can check by \textsf{GAP} that no such non-solvable group satisfies $33\leq c(G)\leq 48$.
    
    \textbf{Subcase~-~4~}If $H=\Z_p$, where $p\geq 7$, then $|G|=60p$. In this case, if $G$ has a subgroup of order $30$, then $G$ has a subgroup of order $30p$, which is a subgroup of the smallest prime index in $G$, hence normal in $G$ by~\cite{dummit1991abstract}. This implies that $G$ is solvable, which is a contradiction. Thus, $G$ does not contain a subgroup of order $30$. Moreover, $G$ has a normal subgroup of order $p$, so by Schur-Zassenhaus Theorem~\cite{hans}, $G$ has a subgroup of order $60$ isomorphic to $A_5$ or $\Z_5\times A_4$. If the subgroup of order $60$ in $G$ is normal, then $G=M\times \Z_p$, where $M=A_5$ or $\Z_5\times A_4$. In both cases, $c(G)=64$ and $32$ respectively, which is not possible. Thus, the subgroup of order $60$ is not normal in $G$, and any subgroup of order $60$ has $p$ conjugates in $G$. If $G$ has $\Z_5\times A_4$ as a subgroup of order $60$, then $G$ has $p$ subgroups isomorphic to $\Z_5\times A_4$. By taking different intersection possibilities, and by using the fact that $G/H=A_5$, we can check that $c(G)>48$. Now, we are left with the case when the subgroup of order $60$ in $G$ is $A_5$, which is not unique. Let $P$ and $Q$ be the two subgroups of $G$ isomorphic to $A_5$. Then, by taking all the possibilities for $P\cap Q$, we see that $c(G)>48$. 
Combining both cases, we see that $G$ has no non-trivial proper normal subgroup, that is, $G$ is simple. As there is no simple group $G$, with $33\leq c(G)\leq 48$. Hence, the result holds.
\end{proof}
\begin{theorem}\label{SG4}
    Let $G$ be a non-solvable group such that $c(G)=49$, then $G\cong SL(2,5)$. 
\end{theorem}
\begin{proof} From classification of finite simple groups and Lemma~\ref{lem:simple2}, it follows that $G$ is not simple.

\textbf{Claim 1:} $G$ is a perfect group. If $G$ is not perfect, then $\{e\}\subsetneq G'\subsetneq G$. Since $G/G'$ is abelian, then $G'$ is non-solvable. As $c(G')<c(G)$, by Theorem~\ref{SG1} and \ref{SG3} $c(G')=32$ and $G'=A_5$. Also, $G/G'=G/A_5$ is a non-trivial group. Let $K/A_5$ be a subgroup of $G/A_5$ of prime order $p$. If $p=2,3$ or $5$, then $K$ is a non-solvable group of order $120,180$ or $300$. By \textsf{GAP}~\cite{GAP4}, one can check that $K=SL(2,5)$ but $SL(2,5)$ has no subgroup isomorphic to $A_5$, so this is not possible. Therefore, $p>5$ and $|K|=60p$ and the number of Sylow $p$ subgroups of $K$, $n_p(K)=1+kp|60$. If the Sylow $p$-subgroup of $K$ is not normal, then $p=7,11,19$ and $29$. For $p=19$ and $29$, it is easy to check that $c(G)>49$. Moreover, for $p=7$ and $11$, it can be checked by \textsf{GAP}~\cite{GAP4} that there is no non-solvable group that satisfies $c(G)=49$. Thus, the Sylow $p$-subgroup $P$ of $K$ is normal and $K=A_5\times P$, also $c(K)=c(A_5)\times c(P)=32\times 2=64$, which is a contradiction. Therefore, $G$ is a perfect group. Since $G$ is not a simple group, $G$ has a proper non-trivial normal subgroup, say $N$.

\textbf{Claim 2:} The next claim is, $N$ is a solvable group. If $N$ is non-solvable, then by the Theorems~\ref{SG1} and~\ref{SG3}, we have $c(N)=32$ and $N= A_5\trianglelefteq G$. Let $K/N$ be a subgroup of prime order $p$ in $G/N$. Then $|K|=60p$. Now, same as the proof of Claim $1$, we can show that this case is not possible. Thus, $N$ is solvable.\\ Again, $G/N$ is non-solvable and $G/N=A_5$. If there exists a proper subgroup $M$ of $N$, which is normal in $G$, then by a similar argument, we get $G/M=A_5$. Thus $|M|=|N|$. Hence, $N$ is a minimal normal subgroup of $G$ and if $M\trianglelefteq G$, then $|M|=|N|$.

\textbf{Claim 3:} $N$ is the unique non-trivial proper normal subgroup of $G$. Suppose $M$ is another non-trivial proper normal subgroup of $G$. Then $|M|=|N|$. Moreover, $MN$ is a normal subgroup of $G$ with $|MN|>|N|$. Thus $MN=G$, and $$|MN|=\frac{|M||N|}{|M\cap N|}=\frac{|N|^2}{|M\cap N|}=|G|=60|N|, i.e. |M\cap N|=|N|/60.$$ Since $N$ is a minimal normal subgroup, we have $|M\cap N|=1$, $|N|=60$ and $|G|=3600$. Now, using \textsf{GAP}, one can check that among all perfect groups $G$ of order $3600$, none of them satisfies $c(G)=49$. Thus, claim $3$ holds.\\
Now, one can also observe that $N$ is characteristically simple because if not, let $\{e\}\lneq M\lneq N$ be a characteristic subgroup of $N$. Then $M$ is a normal subgroup of $G$ properly contained in $N$, which is not possible. Since $N$ is solvable and characteristically simple, $N=(\Z_p)^n$.\\
 If $n\geq 5$, then $c(N)\geq 32$ and $c(G)\geq c(G/N)+c(N)-1=63$, a contradiction. Thus $n\leq 4$. If $n=4$, then $c(N)=p^3+p^2+p+2$ and only prime satisfying $c(N)\leq 18$ is $p=2$. If $n=3$, then $c(N)=p^2+p+2$ and only prime satisfying $c(N)\leq 18$ are $p=2$ and $3$. If $n=2$, then $c(N)=p+2$ and only prime satisfying $c(N)\leq 18$ are $p=2,3,5,7,11,13$. By \textsf{GAP}~\cite{GAP4}, we can check that no non-solvable group of these orders contains $49$ cyclic subgroups. Thus, we must have $N\cong \Z_p$ and $|G|=60p$.
 If $p=2,3$ or $5$, then $|G|=120, 180$ or $300$. By \textsf{GAP}~\cite{GAP4}, one can check that $G\cong SL(2,5)$.\\
 Thus $p\geq 7$, also $N\trianglelefteq  G$ and $(|N|,|G|/|N|)=1$. By Schur–Zassenhaus theorem, $N$ has a complement in $G$, i.e., $G$ has a subgroup of order $60$. Next, we claim that $G$ has no subgroup of order $30$. Suppose $S$ is a subgroup of order $30$ in $G$. Then $G$ has a subgroup of order $30p$, which is normal and solvable, and its quotient is also solvable. This shows that $G$ is solvable, which is a contradiction. Thus, the subgroup of order $60$ in $G$ is either $A_5$ or $\Z_5\times A_4$, and any subgroup of order $60$ in $G$ is maximal. If the subgroup of order $60$ is normal in $G$, then $c(G)\not =49$. This implies that the number of conjugates of any $60$ order subgroup in $G$ is exactly $p$. If the subgroup of order $60$ is $A_5$, then by taking all the possibilities of intersection, we get $c(G)>49$. Similarly, we can check when the subgroup of order $60$ is $\Z_5\times A_4$, then also $c(G)>49$. After combining all these arguments and using \textsf{GAP}, we get $c(G)>49$. Hence $G\cong SL(2,5)$.
\end{proof}

\section{Results related to Supersolvability}\label{sec:supersolvability}
In the existing literature, all groups satisfying $1\leq c(G)\leq 12$ have been completely classified. Among these, the only non-supersolvable groups are $A_4$ and $ {\Z_2}^2\rtimes \Z_9=$ {\em SmallGroup}$(36,3)$. In this section, we extend this classification by determining all non-supersolvable groups satisfying $13\leq c(G)\leq 17$.
\begin{lemma}\label{lem:SSGa}
  Let $G$ be a non-supersolvable group of order $p^2qr$, where $p,q$ and $r$ are distinct prime numbers. Then $G\cong \Z_r\times A_4$.  
\end{lemma}
\begin{proof}
  If $|G|=p^2qr$, then the Sylow $p$-subgroup of $G$ is not cyclic, and both Sylow $q$ and $r$-subgroups cannot be normal. This gives two cases, the first is when both Sylow $q$ and $r$-subgroups are not normal. Then by Sylow's theorems, we have $1+(1+p)+(1+q)+(1+r)\leq c(G)\leq 17$, where $1$ stands for the trivial subgroup of $G$. This shows that $p+q+r\leq 13$, and $|G|\in \{60,84,90,126,140,150,294,315,350,490,525,735\}$. By using \textsf{GAP}, any non-solvable group of these orders has more than $17$ cyclic subgroups. Thus, this case is not possible.  The next case is when one of the Sylow $q$ or $r$-subgroups is normal. Without loss of generality, assume that the Sylow $r$-subgroup is normal. Then $G$ has a subgroup of order $qr$. First, assume that a subgroup of order $qr$ is cyclic. Then it can not be normal, so $G$ has either $p$ or $p^2$ cyclic subgroups of order $qr$.
  
  \textbf{Case~-~1~}~First, suppose that $G$ has $p$ cyclic subgroups of order $qr$. Then, $(1+p)+(1+q)+2+p\leq c(G)\leq 17$, here $2$ denotes the trivial subgroup and the unique subgroup of order $r$, and $p$ denotes the number of cyclic subgroups of order $qr$. This shows that $2p+q\leq 13$, which further shows that if $p=2$ or $3$, then $q\leq 7$ and for $p=5$, we have $q\leq 3$. Now we will see all these possibilities separately.
  
\textbf{Subcase~-~1~}~First take $p=2$ and $q\leq 7$. If the Sylow $2$-subgroup is not normal, then $G$ has at least $7$ cyclic subgroups of order $2$. Also, by Sylow theorem $n_q\in \{4,r,2r,4r\}$. If $n_q\geq r$, then $7+r+2+2\leq c(G)\leq 17$, where the first $2$ denotes the trivial subgroup and the subgroup of order $r$, and the last $2$ denotes the number of cyclic subgroups of order $qr$. This implies that $r\leq 5$ and $|G|= 60, 84$ or $140$. By \textsf{GAP}, we can explicitly check that no such group of these orders has at most $17$ cyclic subgroups. If $n_q=4$, then $q=3$. In this case, by taking different possibilities for the number of cyclic subgroups of order $2,6$ and $2r$ and by using Lemma~\ref{lem:2.2}, we can check that no such group has $17$ cyclic subgroups. If the Sylow $2$-subgroup is normal in $G$, then $G$ has a unique subgroup of order $4r$, which further has $3$ cyclic subgroups of order $2r$. For $n_q\geq r$, we have $c(G)\geq 3+r+2+2+3$, where the first $2$ denotes the number of cyclic subgroups of order $qr$, the second $2$ denotes the trivial subgroup and the subgroup of order $r$, and $3$ denotes the number of cyclic subgroups of order $2r$. Since $c(G)\leq 17$, so $r\leq 7$, and $|G|=60,84,140$. By \textsf{GAP}, we can check that no such group of these orders exists. If $n_q=4$, then $q=3$ and $c(G)\geq 3+4+2+2+3$, where first $2$ denotes the number of cyclic subgroups of order $1$ and $r$, second $2$ denotes the number of cyclic subgroups of order $3r$, first $3$ denotes the number of cyclic subgroups of order $2$ and last $3$ denotes the number of cyclic subgroups of order $2r$. Now, we can consider other possibilities for the number of cyclic subgroups of order $ 6, 2r$, and $6r$, and, using Lemma~\ref{lem:2.2}, we see that none of these cases is possible.

\textbf{Subcase~-~2~}~If $p=3$ and $q\leq 7$, then $n_q\in \{3,9,r,3r,9r\}$. If the Sylow $3$-subgroup is not normal, then $G$ has at least $10$ cyclic subgroups of order $3$. If $n_q\geq r$, then $10+r+2+3\leq c(G)\leq 17$, where $2$ denotes the trivial subgroup, and a subgroup of order $r$, and $3$ denotes the number of cyclic subgroups of order $qr$. This implies that $r=2$ and $|G|=90$ or $126$. By \textsf{GAP}, we can check that no such group of these orders satisfies $13\leq c(G)\leq 17$. If $n_q=9$, then $10+9+2+3\leq c(G)\leq 17$, which is a contradiction. If $n_q=3$, then $10+3+2+3\leq c(G)$, where $2$ denotes the number of cyclic subgroups of order $1$ and $r$, and last $3$ denotes the number of cyclic subgroups of order $qr$. Since $c(G)\leq 17$, this case is also not possible. Now we are left with the case when the Sylow $3$-subgroup is normal in $G$, then $G$ has a unique abelian subgroup of order $9r$, which further has $4$ cyclic subgroups of order $3r$. For $n_q\geq r$, we have $c(G)\geq 4+r+2+3+4$, where the first $4$ denotes the number of cyclic subgroups of order $3$, $2$ denotes the trivial subgroup and the subgroup of order $r$, $3$ denotes the number of cyclic subgroups of order $qr$, last $4$ denotes the number of cyclic subgroups of order $3r$. This gives $r\leq 4$ and $|G|=90, 126$. If $n_q=9$, then $c(G)\geq 4+9+2+3\geq 18$, which is not possible. If $n_q=3$, then $c(G)\geq 4+3+2+3+4$, where the first $4$ denotes the number of subgroups of order $3$, $2$ denotes the trivial subgroup, and the subgroup of order $r$, second $3$ denotes the number of cyclic subgroups of order $qr$ and the last $4$ denotes the number of cyclic subgroups of order $3r$. In this case, either $c(G)=16$ or $G$ has one more cyclic subgroup of order $6$. By using Lemma~\ref{lem:2.2}, we can check that these cases are not possible.

    \textbf{Subcase~-~3~}~Now consider $p=5$ and $q\leq 3$. Also by Sylow theorem $n_q\in \{5,25,r,5r,25r\}$. Since $c(G)\leq 17$, so $n_q=25$ is not possible. First, assume that the Sylow $5$-subgroup is not normal in $G$. Then $G$ has at least $16$ cyclic subgroups of order $5$ and $3$ cyclic subgroups of order $1,q$ and $r$, so $c(G)\geq 19$. This is a contradiction. Therefore, the Sylow $5$-subgroup is normal in $G$, so $G$ has $6$ cyclic subgroups of order $5$. Also, $G$ has an abelian subgroup of order $25r$, which further contains $6$ cyclic subgroups of order $5r$. Thus $c(G)\geq 6+n_q+2+5+6$, where the first $6$ denotes the number of cyclic subgroups of order $5$, the last $6$ denotes the number of cyclic subgroups of order $5r$, $5$ denotes the number of cyclic subgroups of order $qr$ and $2$ denotes the trivial subgroup and the subgroup of order $r$, which is not possible.
    
\textbf{Case~-~2~}~Now assume that $G$ has $p^2$ cyclic subgroups of order $qr$. Thus, by Sylow's theorem $(1+p)+(1+q)+2+p^2\leq c(G)$, where the first and second terms denote the number of cyclic subgroups of order $p$ and $q$, respectively, $2$ denotes the number of cyclic subgroups of order $1$ and $r$. Since $c(G)\leq 17$, then $p^2+p+q\leq 13$. This implies that $p=2$ and $q\leq 7$. If a Sylow $2$-subgroup is not normal, then $c(G)\geq 7+n_q+2+4$, where $7$ denotes the number of cyclic subgroups of order $2$, $2$ denotes the trivial subgroup, and the subgroup of order $r$, and $4$ denotes the number of cyclic subgroups of order $qr$. Since $n_q\geq 4$, then either $c(G)=17$ with $q=3$ or $c(G)>17$. By using Lemma~\ref {lem:2.2}, we can say that this case is not possible. Now, assume that the Sylow $2$-subgroup is normal in $G$, then $G$ has an abelian subgroup of order $4r$, which contains $3$ cyclic subgroups of order $2r$. Thus, $c(G)\geq 3+n_q+2+3+4$, where the first $3$ denotes the number of cyclic subgroups of order $2$, $2$ denotes the trivial subgroup and the subgroup of order $r$, the second $3$ denotes the number of cyclic subgroups of order $2r$, and $4$ denotes the number of cyclic subgroups of order $qr$. If $q\geq 5$, then $n_q\geq 6$, which makes $c(G)>17$. Thus $q=3$ and $n_q=4$. This implies that $c(G)\geq 16$. Also, we can not take $c(G)=17$ as any other possibility for the cyclic subgroup gives at least two cyclic subgroups, which makes $c(G)\geq 18$. Therefore, $c(G)=16$ and $G$ can not have a cyclic subgroup of order $6$. Moreover, $G$ always has a subgroup of order $12$ because $G$ has a normal subgroup of order $4$. Hence, the subgroup of order $12$ in $G$ is $A_4$. If the subgroup of order $12$ is not unique, then they are at least $r\geq 5$, which makes $c(G)>17$. Thus $G$ has a unique subgroup of order $12$ isomorphic to $A_4$. Hence $G=\Z_r\times A_4$.

Now we are left with the case when the subgroup of order $qr$ is not cyclic. Then the subgroup of order $qr$ is isomorphic to $\Z_r\rtimes \Z_q$, so $G$ has at least $r$ subgroups of order $q$. By using this we get $(1+p)+(1+q)+2\leq 17$ and $(1+p)+r+2\leq 17$, where $2$ denotes the trivial subgroup and the subgroup of order $r$. This shows that $p+q\leq 13$ and $p+r\leq 14$. This gives $|G|\in \{90,126,150,198,294,350,490,495,525,726\}$. By using \textsf{GAP}, we can check that such a group does not exist.  
\end{proof}
\begin{lemma}\label{lem:SSGb}
    There does not exist a non-supersolvale group of order $p^3qr$, where $p,q$ and $r$ are distinct prime numbers with $13\leq c(G)\leq 17$.
\end{lemma}
\begin{proof}
 If $|G|=p^3qr$, then a Sylow $p$-subgroup of $G$ is not cyclic. By~\cite[Theorem 1.1]{LowerBoundNumberofCyclicSubgroups} any Sylow $p$-subgroup has at least $5$ cyclic subgroups if $p=2$ and $2p+2$ cyclic subgroups if $p$ is an odd prime. Also, both Sylow $q$ and $r$-subgroups can not be normal in $G$. First, suppose that both the Sylow $q$ and $r$-subgroups are not normal in $G$. Then by Sylow theorem either $5+(q+1)+(r+1)+1\leq 17$ for $p=2$ or $(2p+2)+(q+1)+(r+1)\leq 17$ for $p>2$. In both cases we get $|G|=120,168$ and $270$. By \textsf{GAP}, we can check that these cases are not possible. From now onwards, assume that a Sylow $q$-subgroup is not normal and the Sylow $r$ subgroup is normal in $G$. For $p\geq 5$, any Sylow $p$ subgroup of $G$ has at least $12$ cyclic subgroups by~\cite[Theorem 1.1]{LowerBoundNumberofCyclicSubgroups}. Also, $G$ has a subgroup of order $qr$. If the subgroup of order $qr$ is cyclic, then it is not unique, and $c(G)\geq 12+3+1+2$, where $3$ denotes the number of subgroups of order $q$, $1$ denotes subgroup of order $r$ and $2$ denotes the subgroup of order $qr$. This gives $c(G)\geq 18$, which is a contradiction. If the subgroup of order $qr$ is not cyclic, then it is isomorphic to $\Z_r\rtimes \Z_q$. This shows that $G$ has at least $r$ cyclic subgroups of order $q$. In this case, either $q=2$, $r=3$, and $|G|=750$ or $r\geq 5$, which makes $c(G)\geq 18$. Therefore $p=2$ or $3$. Now we discuss these cases separately.
 
 \textbf{Case~-~1~}~If $p=3$, then $n_q\in \{3,9,27,r,3r,9r,27r\}$. First, assume that the Sylow $p$-subgroup is normal, then the Sylow $p$-subgroup has at least $8$ cyclic subgroups by~\cite[Theorem 1.1]{LowerBoundNumberofCyclicSubgroups}. Also, $G$ has a unique subgroup of order $p^3r$, which further contains $7$ non-trivial cyclic subgroups of order $p^kr$, where $k\geq 1$. Thus, $c(G)\geq 8+3+1+7$, where $3$ denotes the number of subgroups of order $q$ and $1$ denotes a unique subgroup $r$. This implies that $c(G)\geq 19$, this case is not possible. Now, assume that the Sylow $p$-subgroup is not normal, then $G$ has at least $13$ cyclic subgroups of order $p^k, 0\leq k\leq 2$. Thus, $c(G)\geq 13+n_q+1$, where $1$ denotes the cyclic subgroup of order $r$. Since $n_q\geq 1+q$ and $G$ has a subgroup of order $qr$. If the subgroup of order $qr$ is cyclic, then it is not unique. Therefore, $c(G)\geq 13+(1+q)+1+3$, where $3$ denotes the number of cyclic subgroups of order $qr$, which is a contradiction. If a subgroup of order $qr$ is non-cyclic, then $G$ has at least $r$ subgroups of order $q$. This shows that $q=2$ and $r=3$, which is also not possible as $p=3$.
 
 \textbf{Case~-~2~}~If $p=2$, then $n_q\in \{4,8,r,2r,4r.8r\}$. Then we can follow the similar steps as we did for $p=3$ and the possible orders for $G$ are $120,168$ and $280$. Now, using \textsf{GAP} we can check that no such non-supersolvable group of these orders can have at most $17$ cyclic subgroups.   
\end{proof}
\begin{theorem} 
    Let $G$ be a non-supersolvable group such that $13\leq c(G)\leq 17$, then $|G|$ has exactly $2$ distinct prime factors or $G\cong \Z_r\times A_4$, where $r$ is an odd prime number.
\end{theorem}
\begin{proof}
Since $G$ is not supersolvable, neither $G$ is a $p$-group, nor $|G|$ is square-free. First, assume that $|G|$ has at least three prime factors. Then by Richard's theorem~\cite{richards1984remark}, $|G|$ is either $p^2qr$ or $p^3qr$. Now the result is true by Lemma~\ref{lem:SSGa} and~\ref{lem:SSGb}.
\end{proof}
\begin{lemma}\label{lem1SSG}
The only non-supersolvable groups $G$ such that $13\leq c(G)\leq 17$ and $|G|=p^\alpha q$ are $S_4, \Z_2\times A_4, SL(2,3)$ and ${\Z_2}^3\rtimes \Z_7$.
\end{lemma}
\begin{proof}
Let $G$ be a non-supersolvable group of order $p^{\alpha}q$ such that $13\leq c(G)\leq 17$. Then $\alpha \geq 2$ and $q\nmid (p-1)$ (otherwise $G$ will be supersolvable). Since a Sylow $q$-subgroup is cyclic, it can not be normal in $G$. Thus, the number of Sylow $q$-subgroups in $G$ is $n_q=1+qk\geq p^2$. Again, as a Sylow $q$-subgroup is cyclic, a Sylow $p$-subgroup can not be cyclic. Thus by Theorem $1.1$, \cite{LowerBoundNumberofCyclicSubgroups}, for $|G|=2^3q$, $G$ has at least $4$ cyclic subgroups of order $2^k$, where $k=1$ or $2$. Therefore, by counting the number of cyclic subgroups of $G$, we get
\begin{equation}
    1+4+(q+1)\leq 17,
\end{equation}
where $1$ denotes the trivial subgroup. Therefore $q\leq 11$ and $|G|=24, 40, 56$ and $88$. By using \textsf{GAP}, we get $G\cong \Z_2\times A_4, SL(2,3), S_4$ and ${\Z_2}^3\rtimes \Z_7=${\em SmallGroup$(56, 11)$}.\\
If $p^{\alpha}\not =2^3$, then by Theorem $1.1$, \cite{LowerBoundNumberofCyclicSubgroups}, at $G$ has at least $(\alpha-1)p+1$ many cyclic subgroups of order $p^k$, where $1\leq k\leq \alpha-1$. Thus, counting the number of cyclic subgroups of $G$, we get 
\begin{equation}\label{palphaq-eq-1}
    1+(\alpha-1)p+1+n_q\leq 17,
\end{equation}
where the first $1$ denotes the trivial subgroup. This gives, $p^2+p\leq 15$, i.e., $p=2$ or $3$.\\ 
If $p=3$, then $n_q=1+qk=9$ implies $q=2$, i.e., $q|(p-1)$, a contradiction. This implies that $p=2$. Now, $n_q=1+qk=2^2,2^3$ or $2^4$ (as $n_q\geq 2^5$ will make the total number of cyclic subgroups more than $17$) implies that $q=3,5$ or $7$.\\ 
If $q=7$, then $n_q=8$, then by using Equation \ref{palphaq-eq-1}, we get $\alpha\leq 4$. If $q=5$, then $n_q=16$ and hence from Equation \ref{palphaq-eq-1}, $\alpha$ has no solution. If $q=3$, then $n_q=4$ or $16$ and hence from Equation \ref{palphaq-eq-1}, we get $\alpha \leq 6$.\\
Now an exhaustive search on non-supersolvable groups of orders $2^\alpha\cdot 3$ with $2\leq \alpha\leq 6,\alpha\not =3$, and $2^\alpha\cdot 7$ with $\alpha= 2,4$ reveals that no such group $G$ with $13\leq c(G)\leq 17$ exists. Hence, the theorem holds. 
\end{proof}
\begin{lemma}\label{lem2SSG}
    Let $G$ be a non-supersolvable group such that $13\leq c(G)\leq 17$ and $|G|=p^\alpha q ^\beta$. Then $\alpha+\beta\leq 6$.
\end{lemma}
\begin{proof}
By Lemma~\ref{lem1SSG}, it is clear that if $\alpha$ or $\beta$ is $1$ and $G$ is non-supersolvable, then $\alpha+\beta\leq 6$. Now, assume that $\alpha,\beta\geq 2$ and $\alpha+\beta\geq 7$. By using Richard's theorem~\cite{richards1984remark}, it is easy to see that $c(G)\geq 18$. Thus, such $\alpha$ and $\beta$ does not exist.  
\end{proof}
\begin{lemma}\label{lem3SSG}
    The only non-supersolvable group $G$ such that $13\leq c(G)\leq 17$ and $|G|=p^\alpha q^2$, where $\alpha\geq 2$ is $(\mathbb{Z}_2\times \mathbb{Z}_2)\rtimes \mathbb{Z}_{27}$.
\end{lemma}
\begin{proof}Suppose such a group $G$ exists. Then by Lemma~\ref{lem2SSG}, $\alpha=2,3$ or $4$. Now, there are the following cases.

{\bf Case 1:} Let a Sylow $q$-subgroup $S_q$ of $G$ be cyclic. Then $S_q$ is not normal in $G$. Thus the number of Sylow $q$-subgroup of $G$ is $n_q=1+qk\geq p$ and $n_q~|~p^\alpha$. Again, as $S_q$ is cyclic, the Sylow $p$-subgroup $S_p$ is not cyclic. Therefore, we get 
\begin{equation}\label{palphaq2-eq-1}
    c(S_p)+n_q+1\leq 17,
\end{equation}
where $1$ denotes a subgroup of order $q$ in $G$. For $\alpha=4$, by Theorem $1.1$, \cite{LowerBoundNumberofCyclicSubgroups}, we have $c(S_p)\geq 3p+2$. Thus from Equation \ref{palphaq2-eq-1}, we get $(3p+2)+n_q+1\leq 17$. As $n_q=1+qk\geq p$, we get $4p\leq 14$, i.e., $p=2$ or $3$. If $p=2$, then $n_q=1+qk|16$ which implies $q=3,5$ or $7$, and $|G|=2^4\cdot3^2, 2^4\cdot5^2$ and $2^4\cdot7^2$. If $p=3$, then from Equation \ref{palphaq2-eq-1}, we get $11+(1+q)+1=13+q\leq 17$. This implies that $q=2$ and $|G|=3^4\cdot2^2$.\\
For $\alpha=3$, using Equation \ref{palphaq2-eq-1} and arguing as above, we get the only possible orders of $G$ as $2^3\cdot 3^2,2^3\cdot 5^2,2^3\cdot 7^2, 3^3\cdot 2^2, 3^3\cdot 5^2, 3^3\cdot 7^2,5^3\cdot 2^2,5^3\cdot 3^2$.\\
Similarly, for $\alpha=2$, the possible orders are $2^2\cdot3^2,2^2\cdot5^2,2^2\cdot7^2,2^2\cdot{11}^2,3^2\cdot5^2,3^2\cdot7^2,5^2\cdot7^2$. Now, an exhaustive search on non-supersolvable groups of above orders reveals that no such group $G$ with $13\leq c(G)\leq 17$ exists.

{\bf Case 2:} Now, we assume that a Sylow $q$-subgroup $S_q$ is not cyclic. Then a Sylow $p$-subgroup $S_p$ is not both cyclic and normal in $G$. Further, the similar counting arguments as in Case~1 leave only finitely many possible choices for the orders of $G$, which are as follows $2^4\cdot3^2,2^4\cdot5^2,2^4\cdot7^2,3^4\cdot2^2,3^4\cdot5^2,3^4\cdot7^2,5^4\cdot2^2,5^4\cdot3^2,2^3\cdot3^2,2^3\cdot5^2,2^3\cdot7^2,2^3\cdot{11}^2,3^3\cdot2^2,3^3\cdot5^2,3^3\cdot7^2,5^3\cdot2^2,5^3\cdot3^2,5^3\cdot7^2,7^3\cdot2^2,7^3\cdot3^2,7^3\cdot5^2,2^2\cdot3^2,2^2\cdot5^2,2^2\cdot7^2,2^2\cdot{11}^2,3^2\cdot5^2,3^2\cdot7^2,3^2\cdot{11}^2,5^2\cdot7^2$. By using \textsf{GAP} and Sylow theorems, we can check that $G= (\mathbb{Z}_2\times \mathbb{Z}_2)\rtimes \mathbb{Z}_{27}=${\em SmallGroup$(108, 3)$}.
\end{proof}
\begin{lemma}\label{p3q3}
    There does not exist any non-supersolvable group $G$ of order $p^3q^3$ such that $13\leq c(G)\leq 17$.
\end{lemma}
\begin{proof}
We prove this result casewise:
\begin{enumerate}
    \item \textbf{$\mathbf{G}$ has unique subgroups of order $\mathbf{p}$ and $\mathbf{q}$~-~} In this case, either both the Sylow subgroups are cyclic or one is a quaternion group of order $8$ and the other one is cyclic. If both the Sylow subgroups are cyclic, then the group is metacyclic, hence supersolvable. Thus, this case is not possible. If the Sylow $p$-subgroup is $Q_8$ and the Sylow $q$-subgroup is cyclic, then the Sylow $q$-subgroup is not normal. This implies that $G$ has at least $5+(1+q)+2$ cyclic subgroups, where $5$ denotes the number of cyclic subgroups of $Q_8$, $1+q$ denotes the number of cyclic subgroups of order $q^3$, and $2$ denotes the cyclic subgroups of order $q$ and $q^2$. By using the fact $c(G)\leq 17$, we get $q=3,5$ or $7$. If $q=3$, then $|G|=216$. Using \textsf{GAP}, we can check that this case is not possible. If $q=5$, then $G$ has unique Sylow $q$-subgroup. This implies that $G$ is supersolvable, which is a contradiction. For $q=7$, $G$ has $8$ cyclic subgroups of order $7^3$, $5$ cyclic subgroups are there in $Q_8$, and at least $2$ cyclic subgroups of order $7$ and $7^2$. If the subgroup of order $7$ and $7^2$ are unique, then 
    $G$ has $2$ more cyclic subgroups of order $14$ and $98$ because $G$ has a unique cyclic subgroup of order $2$. If $G$ has more than $1$ cyclic subgroups of order $7^2$, then $G$ has either unique cyclic subgroups of order $7$ and $14$ or at least $2$ cyclic subgroups of order $7$. For both possibilities by Lemma~\ref{lem:2.2}, we can see that none of these is possible.
     \item \textbf{$\mathbf{G}$ has a unique subgroup of order $\mathbf{q}$ and $\mathbf{1+p}$ subgroups of order $p$~-~} In this case the Sylow $q$-subgroup is either cyclic or quaternion group $Q_8$. If the Sylow $q$-subgroup is cyclic, then it is not normal, otherwise $G$ will become supersolvable. First, assume that $p\geq 5$, then $G$ has at least $12$ cyclic subgroups of orders $5^k$, where $0\leq k\leq 3$. Also, $G$ has at least $5$ cyclic subgroups of order $q^3$ and at least $2$ cyclic subgroups of order $q$ and $q^2$. This implies that $c(G)>17$, which is a contradiction. Therefore $p=2$ or $3$. If $p=3$, then $G$ has at least $8$ cyclic subgroups of orders $3^k$, where $0\leq k\leq 2$. Moreover, $G$ has at least $1+q$ cyclic subgroups of order $q^3$ and at least $2$ cyclic subgroups of orders $q$ and $q^2$. By using the fact $c(G)\leq 17$, we get $q=2$ or $5$. If $q=5$, then by the Sylow theorem, $G$ has a unique subgroup of order $5^3$, which makes $G$ supersolvable. Therefore, this case is not possible. If $q=2$, then $|G|=2^3.3^3$. By \textsf{GAP}, we can check that this order is not possible. From now onwards, assume that $p=2$. Then $G$ has at least $6$ cyclic subgroups of orders $2^k$, where $0\leq k\leq 2$. Also, by Sylow's theorem $n_q=1+kq|8$, this implies that $q=3$ or $7$. If $q=7$, then $G$ has at least $8$ cyclic subgroups of orders $7^3$ and a unique cyclic subgroup of order $7$. Also, $G$ has at least one cyclic subgroup of order $7^2$. This implies that $G$ has a unique cyclic subgroup of order $2.7$ or $c(G)>17$. By counting the elements of $G$, we can show that this is not possible. If $q=3$, then $|G|=2^3.3^3$. By \textsf{GAP}, we can check that this case is not possible. Similarly, we can show that if the Sylow $q$ subgroup is isomorphic to $Q_8$, then either $|G|=2^3.3^3$ or $c(G)>17$. By \textsf{GAP}, this case is also not possible.
     \item \textbf{Subgroups of orders $\mathbf{p}$ and $\mathbf{q}$ are not unique~-~} In this case both the Sylow subgroups are non-cyclic. If either $p$ or $q\geq7$, then by \textsf{GAP} we can check that $c(G)>17$. First suppose that $p=5$, then $G$ has at least $12$ cyclic subgroups of orders $p^k$, where $0\leq k\leq 3$. If $q=2$, then $G$ has at least $6$ cyclic subgroups of orders $q^k$, where $0\leq k\leq 3$. Thus, $c(G)>17$. Similarly we can show that if $p\geq 5$ and $q=3$, then $c(G)>17$. Now, we are left with $|G|=2^3.3^3$, and by \textsf{GAP}, we can check that this order is not possible.
    \end{enumerate}
\end{proof}
The following result follows from the lemmas and theorems established in this section.

\begin{theorem}\label{thm:SSG}
    Let $G$ be a group such that $13\leq c(G)\leq 17$, then either $G$ is supersolvable or $G\cong SL(2,3), S_4, \Z_q\times A_4$, ${\Z_2}^3\rtimes \Z_7$ or $(\mathbb{Z}_2\times \mathbb{Z}_2)\rtimes \mathbb{Z}_{27}$, where $q$ is a prime number.
\end{theorem}
\section*{Acknowledgements} 
 The second-named author would like to acknowledge the support of the IISER Berhampur institute post-doctoral fellowship during this work.

\subsection*{Data Availability Statements}
Data sharing is not applicable to this article as no datasets were generated or analyzed during the current study.

\subsection*{Competing Interests} The authors have no competing interests to declare that are relevant to the content of this article.
\bibliographystyle{abbrv}
\bibliography{refs.bib} 
\end{document}